\documentclass[11pt,a4paper]{article}

\usepackage{amsmath,amsfonts,amssymb,latexsym,graphics,epsfig,url}
\usepackage{color}
\usepackage{amsthm}
\usepackage[english]{babel}
\usepackage{graphicx}

\newcommand\blfootnote[1]{%
  \begingroup
  \renewcommand\thefootnote{}\footnote{#1}%
  \addtocounter{footnote}{-1}%
  \endgroup
}

\newtheorem{theorem}{Theorem}[section]
\newtheorem{lemma}[theorem]{Lemma}
\newtheorem{corollary}[theorem]{Corollary}

\DeclareMathOperator{\Aut}{Aut}
\DeclareMathOperator{\Cay}{Cay}

\DeclareMathOperator{\ev}{ev}

\DeclareMathOperator{\tr}{tr}

\DeclareMathOperator{\spec}{sp}

\def\Cay{\mbox{\rm Cay}}

\def\x{\mbox{\boldmath $x$}}

\def\vecrho{\mbox{\boldmath $\rho$}}

\def\vecphi{\mbox{\boldmath $\phi$}}
\def\vec0{\mbox{\boldmath $0$}}

\def\A{\mbox{\boldmath $A$}}
\def\B{\mbox{\boldmath $B$}}
\def\CC{\mbox{\boldmath $C$}}

\def\D{\mbox{\boldmath $D$}}

\def\G{\Gamma}

\def\U{\mbox{\boldmath $U$}}

\def\C{\mathbb C}

\def\G{\Gamma}

\begin{document}

\title{The spectra of lifted digraphs
\thanks{Research of the first two authors is supported by MINECO under project MTM2014-60127-P, and by AGAUR under project 2014SGR1147.
The third author acknowledges support from the research grants APVV 0136/12,  APVV-15-0220, VEGA 1/0026/16, and VEGA 1/0142/17. }}
\author{C. Dalf\'o\footnote{Departament de Matem\`atiques, Universitat Polit\`ecnica de Catalunya, Barcelona, Catalonia, {\tt{cristina.dalfo@upc.edu}}},
M. A. Fiol\footnote{Departament de Matem\`atiques, Universitat Polit\`ecnica de Catalunya; and Barcelona Graduate School of Mathematics, Barcelona, Catalonia, {\tt{miguel.angel.fiol@upc.edu}}}, J. \v{S}ir\'a\v{n}\footnote{Department of Mathematics and Statistics,
The Open University, Milton Keynes, UK; and
Department of Mathematics and Descriptive Geometry,
Slovak University of Technology, Bratislava, Slovak Republic,
{\tt {j.siran@open.ac.uk}}}
}


\maketitle

\blfootnote{
\begin{minipage}[l]{0.3\textwidth} \includegraphics[trim=10cm 6cm 10cm 5cm,clip,scale=0.15]{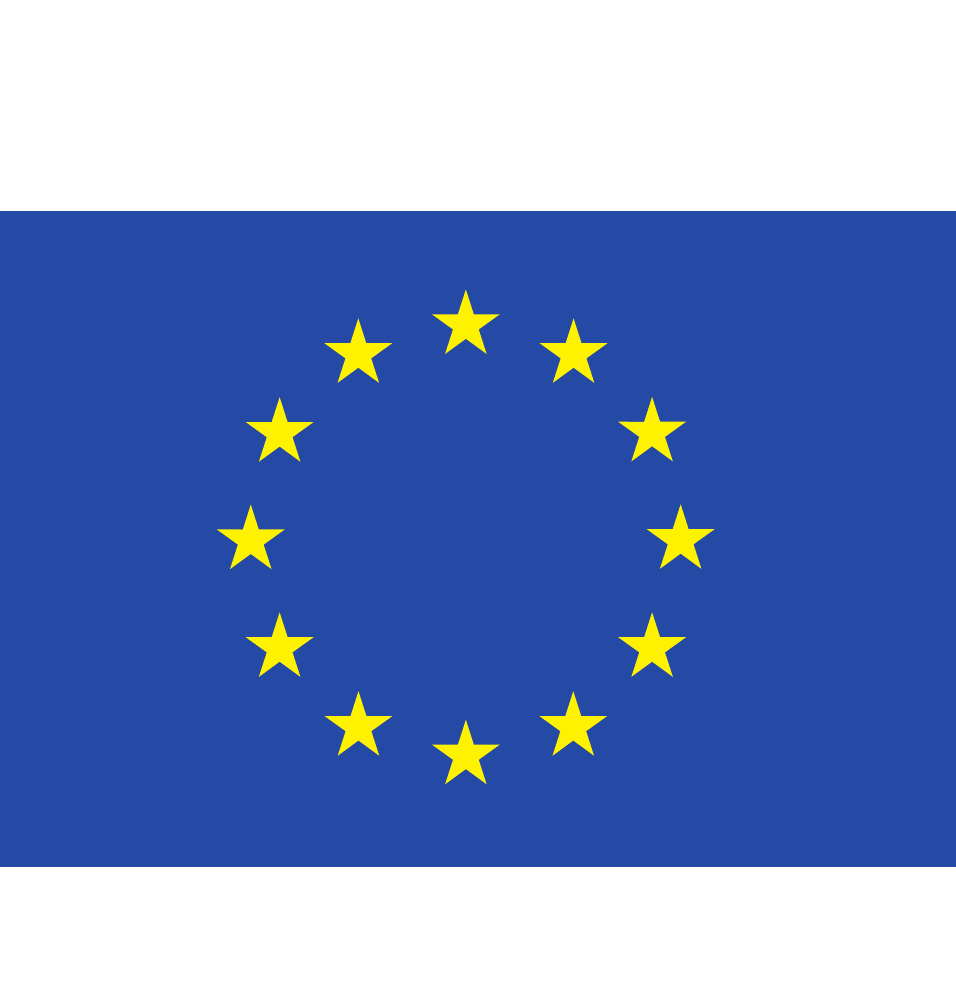} \end{minipage}  \hspace{-2cm} \begin{minipage}[l][1cm]{0.79\textwidth}
   The first author has also received funding from the European Union's Horizon 2020 research and innovation programme under the Marie Sk\l{}odowska-Curie grant agreement No 734922.
  \end{minipage}}

\begin{abstract}
We present a method to derive the complete spectrum of the lift $\Gamma^\alpha$ of a base digraph $\Gamma$, with voltage assignments on a (finite) group $G$.
The method is based on assigning to $\Gamma$  a quotient-like matrix whose entries are elements of the group algebra $\C[G]$, which fully represents $\Gamma^{\alpha}$. This allows us to derive the eigenvectors and eigenvalues of the lift in terms of those of the base digraph and the irreducible characters of $G$.
Thus, our main theorem generalize some previous results of Lov\'az and Babai concerning the spectra of Cayley digraphs.
\end{abstract}



\section{Preliminaries}

As it is well-known, the spectrum of a graph $\G$ (that is, the eigenvalues of its adjacency matrix) is an important invariant that gives us interesting results about its combinatorial properties, see e.g. the classic textbook of  Cvetkovi\'c, Doob, and Sachs \cite{cds95}.
Thus many efforts have been devoted to derive the spectrum (totally or partially) of some interesting families of graphs.
In this framework, Lov\'asz \cite{l75} provided
a formula which expresses the eigenvalues of $\G$ in terms of the characters of a transitive subgroup of its automorphism group $\Aut \G$. In the particular case of Cayley graphs (when the automorphism group contains a subgroup $G$ acting regularly on the vertices),
 Babai \cite{ba79} derived a more handy formula is by different methods. In fact, this formula also holds true for digraphs and arc-colored Cayley graphs. Following these works, here we deal with a more general family of (di)graphs, which are obtained as a type of compounding between a `base digraph' and a Cayley digraph. Our study not only gives the complete spectrum of the obtained digraphs (called `lifts'), but also shows how to compute the corresponding eigenvectors.

 Through this paper, $\Gamma=(V,E)$ denotes a digraph, with vertex set $V$ and arc set $E$. An arc from vertex $u$ to vertex $v$ is denoted by either $(u,v)$, $uv$, or $u\rightarrow v$. We allow {\em loops} (that is, arcs from a vertex to itself), and {\em multiple arcs}.
The spectrum of $\Gamma$, denoted by $\spec \Gamma=\{\lambda_0^{m_0},\lambda_1^{m_1},\ldots,\lambda_d^{m_d}\}$, is constituted by the distinct eigenvalues $\lambda_i$ with the corresponding algebraic multiplicities $m_i$, for $i\in [0,1]$, of its adjacency matrix $\A$. (Throughout the paper, for some integer $n$, we use the notation $[1,n]$ for the set $\{1,\dots,n\}$.)

The paper is organized as follows.  In the rest of this section we recall the definition of a voltage digraph and its lift. Then, we introduce a representation of the
lifted digraph with a quotient-like matrix whose size equals the order of the (much smaller) base digraph. In particular, it is shown that such a matrix can be used to deduce combinatorial properties of the lifted digraph. Following this approach, and as a main result, Section \ref{sec:sp} present a new method to completely determine the spectrum of the lift by using only the spectrum of the quotient-like matrix and the (irreducible) characters of the group. The results are illustrated by following some examples.

For the concepts and/or results about graphs and digraphs not presented here, we refer the reader to some of the basic textbooks on the subject; for instance, Bang-Jensen and Gutin \cite{bg09},  or Diestel~\cite{d10}.

\subsection{Voltage and lifted digraphs}
\label{sec:voltage}

Voltage (di)graphs are, in fact, a type of compounding that consists of connecting
together several copies of a (di)graph by setting some (directed) edges between any
two copies.  More precisely,
let $\G$ be a digraph with vertex set $V=V(\G)$ and arc set $E=E(\G)$.
Then, given a group $G$ with generating set $\Delta$, a
{\em voltage assignment} of $\G$ is a mapping $\alpha:E\rightarrow \Delta$. The pair $(\G,\alpha)$ is often called a {\em voltage digraph}. The {\em lifted digraph} (or, simply, {\em lift})
$\Gamma^\alpha$ is the digraph with vertex set  $V(\Gamma^\alpha)=V\times G$ and
arc set $E(\Gamma^\alpha)=E\times G$, where there is an arc from vertex $(u,g)$ to vertex $(v,h)$ if and only if $uv\in E$ and $h=g\alpha(uv)$.
For example, Figure \ref{fig1}$(b)$ shows the lifted digraph $\G^{\alpha}$ for the base digraph $\G=K_2^*$ with voltages shown in Figure \ref{fig1}$(a)$. More precisely, $\G$ is a complete graph on two vertices $V(\G)=\{a,b\}$, with additional arcs $ab$, $ba$, $aa$, $bb$, and  voltages $\alpha(aa)=\alpha(bb)=\sigma$ and $\alpha(ab)=\alpha(ba)=\rho$,  on the group $G=S_3\cong D_3=\langle \rho, \sigma| \rho^3=\sigma^2=(\rho\sigma)^2=\iota\rangle$.
Notice that because of the group role, the symmetry of the obtained lifts usually yields  digraphs with large automorphism groups.
In particular, when $\G$ is a singleton with loops, the lift is the Cayley digraph $\Cay(G,\Delta)$.
For more information, see  the authors' paper \cite{dfmrs17}, or the comprehensive survey of Miller and \v{S}ir\'{a}\v{n}~\cite{ms}.

\subsection{A matrix representation of the lift}
Let us see how we can fully represent a lifted digraph with a matrix whose size equals the order of the base digraph. This approach was initiated by the auhtors, together with  Miller and Ryan, in \cite{dfmrs17}.
Let $\G=(V,E)$ be a digraph with voltage assignment $\alpha$ on the group $G$. Its
\emph{associated matrix} $\B$ is a square matrix indexed by the vertices of $\G$, and whose entries are elements of the group algebra $\C[G]$. Namely,
$$
(\B)_{uv}=\sum_{g\in G} \alpha_g g
$$
where
$$
\alpha_i=\left\{
\begin{array}{ll}
1 & \mbox{if  $uv\in E$ and $\alpha(uv)=g$,}\\
0 & \mbox{otherwise,}
\end{array}
\right.
$$
for $i\in [1,n]$. The following result was given in \cite{dfmrs17}.
\begin{lemma}
\label{lema-walks}
Let $(\B^{\ell})_{uv}=\sum_{g\in G} a_g^{(\ell)} g$. Then, for every $g,h\in G$,
the coefficient $a_g^{(\ell)}$ equals the number of walks of length $\ell$ in the lifted digraph $\G^{\alpha}$,  from vertex $(u,h)$ to vertex $(v,hg)$.
In particular, if $u=v$ and $\iota$ denotes the identity element of $G$,
$a_{\iota}^{(\ell)}$ is the number of walks of length $\ell$ rooted at every vertex $(u,g)$, for $g\in G$, of the lift.
\end{lemma}

\subsection{Some theoretical background}

In our study we use representation theory of finite groups. For basics on representation theory and character tables of a group, see for instance, Burrow \cite{b93}.

We also  need to recall the following known result (see e.g.  Gould \cite{go99}.).

\begin{lemma}
\label{gould}
If the power sums $s_k=\sum_{i=1}^d z_i^k$ of some complex numbers $z_1,\ldots,z_d$ are known for every $k=1,\ldots,d$, then such numbers are
the roots of the monic polynomial
$$
p(z)=\frac{1}{d!}\det \CC(z),
$$
where $\CC(z)$ is the following matrix of dimension $d+1$:
$$
\CC(z)=
\left(
\begin{array}{cccccccc}
z^d & z^{d-1} & z^{d-2} & z^{d-3} & \cdots & z^2 & z & 1 \\
s_1 & 1       &   0     &    0    &  \cdots & 0 & 0 & 0 \\
s_2 & s_1     &   2     &    0    &  \cdots & 0 & 0 & 0 \\
s_3 & s_2     &   s_1   &    3    &  \cdots & 0 & 0 & 0 \\
\vdots & \vdots     &   \vdots   &    \vdots    &  \ddots & \vdots & \vdots & \vdots \\
s_{d-1} & s_{d-2}     &   s_{d-3}   &    s_{d-4}    &  \cdots & S_1 & d-1 & 0 \\
s_d     & s_{d-1} & s_{d-2}     &   s_{d-3}   &     \cdots & s_2 & s_1 & d \\
\end{array}
\right).
$$
\end{lemma}

\section{The spectrum}
\label{sec:sp}
The following result allows us to compute the spectrum of a lifted digraphs from its associated matrix and the irreducible representations of the group.

\begin{theorem}
\label{theo-sp}
Let $\G=(V,E)$ be a base digraph on $r$ vertices, with a voltage assignment $\alpha$ in a group $G$ with $|G|=n$. Assume that $G$ has $\nu$ conjugacy classes with dimensions $d_1,\ldots,d_{\nu}$ $($so, $\sum_{i=1}^{\nu} d_i^2=n$$)$. Let $\vecrho_1,\ldots,\vecrho_{\nu}$ be the irreducible representations of the group $G$. Let $\vecrho_i(\B)$ the complex matrix obtained from $\B$ by replacing each $g\in G$ by the $d_i\times d_i$ matrix $\vecrho_i(g)$, and  let $\mu_{u,j}$, $u\in V$, $j\in[1,d_i]$ denote its eigenvalues. Then, the $rn$ eigenvalues of the lift $\G^{\alpha}$ are the $rd_i$ eigenvalues of $\vecrho_i(\B)$, for every $i\in [1,\nu]$, each repeated $d_i$ times.
\end{theorem}

\begin{proof}
Let $\A$ be the adjacency matrix of the lift.
First, we prove that, for every $u\in V$ and $j\in[1,d_i]$, every eigenvalue $\mu_{u,j}$ of $\vecrho_i(\B)$ gives rise to $d_i$  (equal) eigenvalue of $\A$. To this end, let $\U_i$ the $rd_i\times rd_i$ matrix whose columns are the eigenvectors of $\vecrho_i(\B)$.
Let $\D_i$ be the diagonal matrix having such eigenvalues as its entries. For every $u\in V$, let $\x_u$ be the $d_i\times rd_i$ submatrix of $\U_i$ having files indexed with $(u,j)$, $j\in[1, d_i]$.
Then, from $\vecrho_i(\B)\U_i=\U_i\D_i$ we have
\begin{equation}
\label{eq-ev-quotient}
\sum_{uv\in E}\vecrho_i(\alpha(uv))\x_v= \x_{u}\D_i\qquad \mbox{for $u\in V$}.
\end{equation}
Now for each vertex $(u,g)$ of $\G^{\alpha}$, consider the $d_i\times rd_i$ matrix
$$
\vecphi_{(u,g)}=\vecrho_i(g)\x_u.
$$
Then,
\begin{align*}
\vecphi_{(u,g)}\D_i &=\vecrho_i(g)\x_u\D_i= \vecrho_i(g)\sum_{uv\in E}\vecrho_i(\alpha(uv))\x_v\\
 &=
\sum_{uv\in E}\vecrho_i(g\alpha(uv))\x_v=
\sum_{uv\in E}\vecphi_{(v, g\alpha(uv)))}.
\end{align*}
But this means that, for every pair of fixed elements $k\in [1,d_i]$ and $(u,j)\in\{V\times [1,d_i]\}$, the vector obtained by taking the
 $(k,(u,j))$-entry of every matrix $\vecphi_{(v,h)}$, for $(v,h)\in V(\G^{\alpha})$, is an eigenvector of the lift $\G^{\alpha}$ with eigenvalue $\mu_{u,j}$. Since there are $d_i$ possible choices for $k$, the same holds for the eigenvectors of $\mu_{u,j}$.

Moreover, by Lemma \ref{lema-walks}, if
$(\B^{\ell})_{uu}=\sum_{g\in G}a_{(u,g)}^{(\ell)}g$,
the total number of rooted closed $\ell$-walks in $\G^{\alpha}$ is
$$
\tr (\A^\ell)=\sum_{\lambda\in \spec \G^{\alpha}} \lambda^\ell=n\sum_{u\in V} a_{(u,\iota)}^{(\ell)}
$$
since, in the lift, the number of $(u,g)$-rooted closed $\ell$-walks does not depend on $g\in G$.
Moreover, by the `Great Orthogonality Theorem' (see, e.g. \cite{b93}), we have $\sum_{g\in G}\vecrho_i(g)=\vec0$ for every $i\neq 1$ (with $\vecrho_1$ being the trivial representation).
Then,
$$
a_{(u,\iota)}^{(\ell)}=\frac{1}{n}\sum_{i=1}^{\nu} d_i(\vecrho_i(\B^{\ell}))_{uu}
$$
and, hence,
$$
\sum_{\lambda\in \spec \G^{\alpha}} \lambda^\ell=\tr (\A^\ell)=\sum_{i=1}^{\nu}d_i\tr (\vecrho_i(\B^{\ell}))
=\sum_{i=1}^{\nu}\sum_{\mu\in \spec \vecrho_i(\B^{\ell})} d_i\mu^{\ell}.
$$
(Note that, in the sum on the right, we have $r\sum_{i=1}^{\nu} d_i^2= rn$ terms, which is the number of eigenvalues of the adjacency matrix $\A$ of the lift $\G^{\alpha}$.)
As the above equality holds for every $\ell=1,2,\ldots$, by Lemma \ref{gould} both (multi)sets of eigenvalues must coincide.
\end{proof}

As a consequence, we have the following result in terms of the (irreducible) characters $\chi_i(g)$, for $g\in G$, of the group.
For stating it, let us introduce some additional notation:
Let $P_{\ell}$ be the set of closed walks of length $\ell\ge 1$ in $\Gamma$. If $p\in P_{\ell}$, say $u_0\rightarrow u_1\rightarrow \cdots \rightarrow u_{\ell-1}\rightarrow u_{\ell}(=u_0)$, let
$$
\chi_i(p)=\prod_{j=0}^{\ell-1}\chi_i(\alpha(u_ju_{j+1})).
$$

\begin{corollary}
\label{coro-chi}
Using the same notation as above, for each $i\in[1,\nu]$, the eigenvalues $\lambda_{u,j}$, for $u\in V$ and $j\in[1,d_i]$, of the lift $\G^{\alpha}$, are the solutions $($each repeated $d_i$ times$)$ of the system
\begin{equation}
\label{babai-gen}
\sum_{u\in V,j\in[1,d_i]}\lambda_{u,j}^{\ell}=\sum_{p\in P_{\ell}}\chi_i(p),\qquad \ell=1,\ldots, rd_i.
\end{equation}
\end{corollary}
\begin{proof}
By Theorem  \ref{theo-sp}, the above left sum of the powers is $\tr(\vecrho_i(\B)^{\ell})$, whereas the  right expression  corresponds to
$\chi_i(\tr(\B^{\ell}))$ (where the $\ell$-power of $\B$ and its trace is computed in $\C[G]$, and $\chi_i(\sum_{g\in G}a_g g)=\sum_{g\in G}a_g\chi_i(g)$). Then, the result follows from
$$
\tr(\vecrho_i(\B)^{\ell})=\sum_{u\in V}\tr[\vecrho_i((\B^{\ell})_{uu})]=\sum_{u\in V}\chi_i((\B^{\ell})_{uu})=\tr(\chi_i(\B^{\ell})).
$$
\end{proof}
Notice that, by Lemma \ref{gould}, the equalities in \eqref{babai-gen} lead to a polynomial of degree $rd_i$, with roots the required eigenvalues $\lambda_{u,j}$.

As commented in Section \ref{sec:voltage}, when $\Gamma$ consists of one vertex with loops, then $\G^{\alpha}$ is a Cayley digraph and \eqref{babai-gen} gives the result of Babai \cite{ba79}.
Another extreme case is when $d_i=1$ for some $i$. Then,  we simply have $\lambda_{u,1}=\mu_{u,1}$ for every $u\in V$. For instance,  when $G$ is Abelian, $\nu=n$, and this holds for every $i\in[1,n]$. This case was dealt with by the authors, Miller, and Ryan in \cite{dfmrs17}.


\subsection{An example}
Let us consider again the lift described in Subsection \ref{sec:voltage} and shown Figure \ref{fig1}. Then, the base graph $K_2^*$ has voltages on the symmetric group $S_3\cong D_3=\langle \rho, \sigma : \rho^3=\sigma^2=(\rho\sigma)^2=\iota\rangle$, with characters shown in Table \ref{charac-table-S3}, and associated matrix
$$
\B=
\left(
\begin{array}{cc}
\sigma & \iota+\rho\\
\iota+\rho & \sigma
\end{array}
\right).
$$
Note that the edge (two opposite arcs forming a digon) of the base digraph gives rise to the entries $\iota$'s for the voltages assigned to the corresponding arcs.

\begin{table}[h]
\centering
\begin{tabular}{|c||c|c|c|}
\hline
$S_3\quad \backslash\quad g$  & $\iota$  & $\sigma,\sigma\rho,\sigma\rho^2$ & $\rho,\rho^2$ \\
\hline\hline
$\chi_1$ $(d_1=1)$  &  $1$  & $1$   &  $1$    \\
\hline
$\chi_2$ $(d_2=1)$  &  $1$  & $-1$   &  $1$   \\
\hline
$\chi_3$ $(d_3=2)$  &  $2$  & $0$   &  $-1$    \\
\hline
\end{tabular}
\caption{The character table of the symmetric group $S_3$.}
\label{charac-table-S3}
\end{table}

The obtained lifted digraph $\G^{\alpha}$  has spectrum
\begin{equation}
\label{sp-example}
\spec \G^{\alpha}=\{3^{(1)},1^{(3)},0^{(4)},-1^{(3)},-3^{(1)}\}.
\end{equation}

\begin{figure}[h]
\begin{center}
\includegraphics[width=12cm]{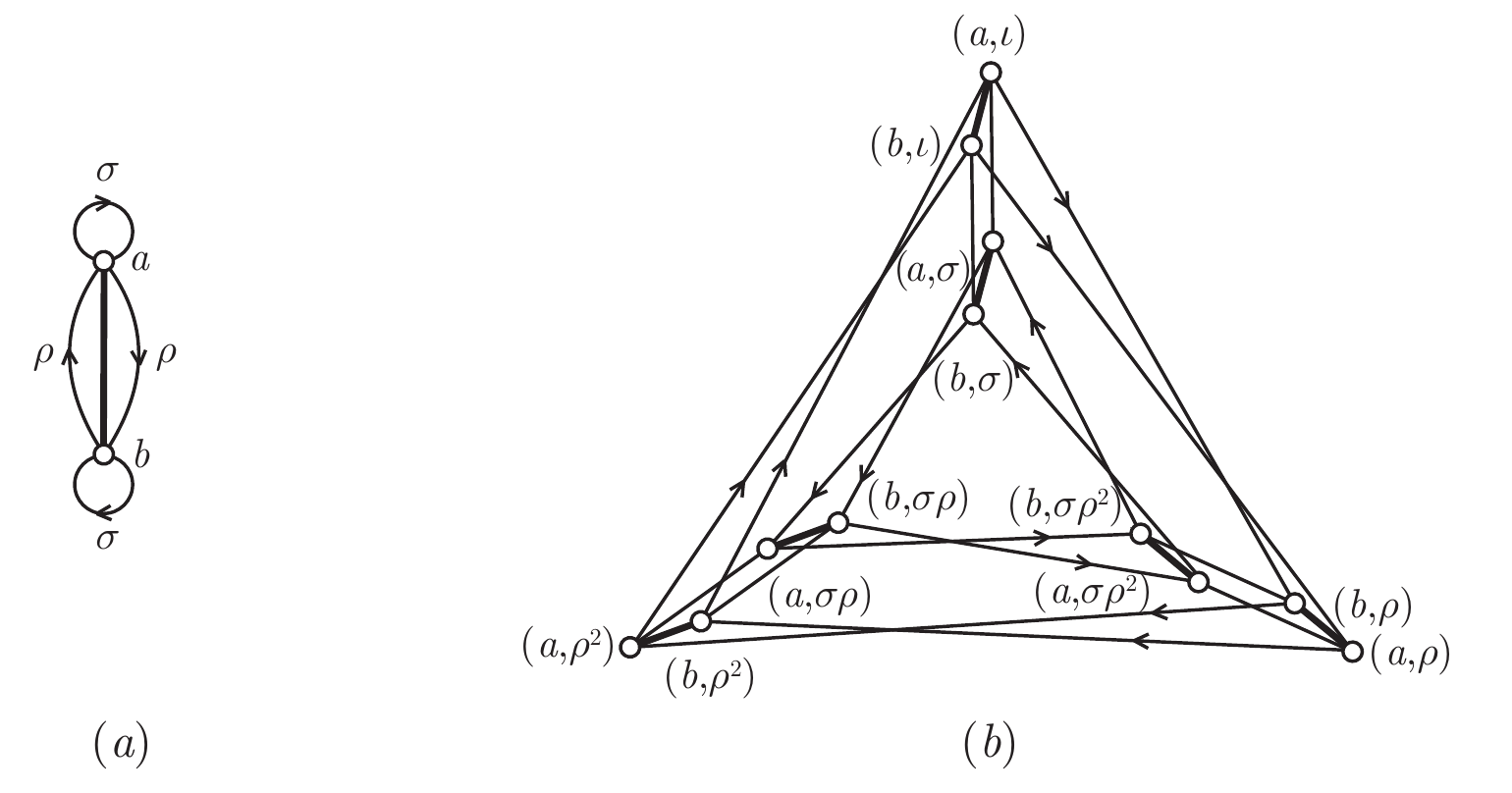}
\end{center}
\vskip -.5cm
\caption{The base digraph $K_2^*$, on the group $S_3$, and its lift}
  \label{fig1}
\end{figure}

Then, we get the complex matrices
$$
\chi_1(\B)=
\left(
\begin{array}{cc}
1 & 2\\
2 & 1
\end{array}
\right),\qquad
\chi_2(\B)=
\left(
\begin{array}{rr}
-1 & 2\\
2 & -1
\end{array}
\right),\qquad
\chi_3(\B)=
\left(
\begin{array}{cc}
0 & 1\\
1 & 0
\end{array}
\right).
$$
Then, by Corollary \ref{coro-chi}:
\begin{itemize}
\item $\chi_1$:
Since $d_1=1$, two  eigenvalues of $\G^{\alpha}$ are
$$
\{3,-1\}=\ev \chi_1(\B).
$$
\item $\chi_2$:
Since $d_2=1$, two  eigenvalues of $\G^{\alpha}$ are
$$
\{-3,1\}=\ev \chi_2(\B).
$$
\item $\chi_3$:
Since $d_3=2$, we consider all the possible closed walks of lengths $\ell=1,2,3,4$ in $\B$, which gives the system
\begin{align*}
\lambda_{u,0}+\lambda_{u,1}+\lambda_{v,0}+\lambda_{v,1} &=0      \\
\lambda_{u,0}^2+\lambda_{u,1}^2+\lambda_{v,0}^2+\lambda_{v,1}^2 &=2\\
\lambda_{u,0}^3+\lambda_{u,1}^3+\lambda_{v,0}^3+\lambda_{v,1}^3 &=0\\
\lambda_{u,0}^4+\lambda_{u,1}^4+\lambda_{v,0}^4+\lambda_{v,1}^4 &=2,
\end{align*}
with solutions $1,0,0,-1$
%
\end{itemize}
Then, as these last eigenvalues have two be considered twice, this completes the  eigenvalue multiset of $\G^{\alpha}$, in agreement with \eqref{sp-example}.



\end{document}